\documentclass[oneside]{amsart}
\usepackage{amsmath}
\usepackage{amsfonts}
\usepackage{amssymb}
\usepackage{latexsym}
\usepackage{amsthm}
\usepackage{diagrams}
\usepackage{setspace}
\usepackage{hyperref}
\usepackage{geometry}
\usepackage{xypic}

\everymath{\displaystyle}
\usepackage{mymacros}
\usepackage{url}

\newcommand{\Sh}{\operatorname{\textbf{Sh}}}
\newcommand{\Supp}{\operatorname{Supp}}

\begin{document}
 \title{A level raising result for modular Galois representations modulo prime powers.}

\author{Panagiotis Tsaknias}
\address{Department of Mathematics, University of Luxembourg,
                          Campus Kirchberg, 6 rue Richard Coudenhove-Kalergi, L-1359 Luxembourg}
                          \email{panagiotis.tsaknias@uni.lu}
		        \email{p.tsaknias@gmail.com}

\date{March 2012}

\begin{abstract}
In this work we provide a level raising theorem for $\mod\la^n$ modular Galois representations. It allows one to see such a Galois representation that is modular of level $N$, weight $2$ and trivial Nebentypus as one that is modular of level $Np$, for a prime $p$ coprime to $N$, when a certain local condition at $p$ is satisfied. It is a generalization of a result of Ribet concerning $\mod \ell$ Galois representations.
\end{abstract}

\maketitle

\section{Introduction}

Let $N$ and $k$ be positive integers, $S_k(\Gamma_0(N))$ be the space of modular forms of level $N$ and weight $k$, and $\T_k(N)$ be the $\Z$-algebra of Hecke operators acting faithfully on this space. Let also $R$ be a complete Noetherian local ring with maximal ideal $\mg_R$ and residue field of characteristic $\ell>0$. A (weak) eigenform of level $N$ and $k$ with with coefficients in $R$ is then defined to be a ring homomorphism $\theta:\T_k(N)\lra R$ (One can find a discussion on the various notions of modularity modulo prime powers as well as a comparison between them in \cite{ChenKimingWiese11}). We will denote by $\bar{\theta}$ its composition with $R\lra R/\mg_R$, i.e. the residual reduction of $\theta$.
Then one has the following theorem of Carayol (Theorem 3 in \cite{Carayol94}):
\begin{thm}[Carayol]
Let $k\geq 2$ and $N>4$ or assume that $6$ is invertible in $R$ (i.e. that $\ell\geq5$). If the representation attached to $\bar{\theta}$ is absolutely irreducible, then one can attach a Galois representation $\rho:\gq\lra\GL_2(R)$ to $\theta$ in the following sense: For every prime $q\nmid N\ell$, $\rho$ is unramified at $q$ and
$$\tr(\rho(\Frob_q)) = \theta(T_q).$$
\end{thm} 
A representation that arises in the way described by the previous theorem is called \emph{modular}. If one wants to explicitly mention a specific eigenform $\theta$ due to which the representation $\rho$ is modular one can say that $\rho$ is \emph{attached to} or \emph{associated with} $\theta$. 

One can then ask if the converse is true: Given a Galois representation $\rho:\gq\lra\GL_2(R)$, when is it modular? Furthermore can one have a hold on what the level and weight of this eigenform will be?

Let $p$ be a rational prime. Then one has a natural inclusion map
$$S_k(\Gamma_0(N))\oplus S_k(\Gamma_0(N)) \lra S_k(\Gamma_0(Np))$$
whose image is called the $p$-old subspace. This subspace is stable under the action of $\T_k(Np)$ and so is its orthogonal complement through the so-called Peterson product. We call this complementary subspace the $p$-new subspace and we denote by $\T_k^{p-\textrm{new}}(Np)$ the quotient of $\T_k(Np)$ that acts faithfully on it. We will call this quotient the $p$-new quotient of $\T_k(Np)$. There is also the $p$-old quotient that is defined in the obvious way.
Fix another rational prime $\ell$. Assume $\Oc$ is the ring of integers of a number field and $\la$ a prime above $\ell$. Here we prove the following level raising result:
\begin{thm}\label{thm:main}
Let $n\geq 2$ be an integer and $\rho : \gq \lra \GL_2(\qtn)$ be a continuous Galois representation that is modular, associated with a Hecke map $\theta:\T_2(N)\lra\qtn$, and residually absolutely irreducible. Let also $p$ be a prime such that $(\ell N,p)=1$ and  assume that $\tr(\rho(\Frob_p)) \equiv \pm (p+1) \mod \la^n$.Then $\rho$ is also associated with a Hecke map $\theta':\T_2(Np)\lra\qtn$ which is new at $p$, i.e. $\theta'$ factors through the $\T_2^{p-\textrm{new}}(Np)$.
\end{thm}

\remark For $n=1$ this is Theorem 1 of \cite{Ribet90a}. 

\remark The theorem does not exclude the case $\ell|N$. 

\remark As with the case $n=1$, one can also prove the theorem in the case $p=\ell$ by assuming the condition $\theta(T_p)\equiv \pm(p+1)\mod\la^n$ instead of the one involving the trace of the representation.

\remark Notice that even if the Hecke map that makes $\rho$ modular in the first place lifts to characteristic 0, i.e. comes from a classical eigenform, there is no guarantee that the Hecke map of level new at $p$ that one obtains in the end lifts too.

\begin{cor}
Let $\rho$ be as in Theorem \ref{thm:main}. Then there exist infinitely many primes $p$ (coprime to $N$) such that $\rho$ is modular of level $Np$, new at $p$.
\end{cor}

\begin{proof}
Immediate consequence of Lemma 7.1 in \cite{Ribet90b}.
\end{proof}

In what follows we set $\T_N := \T_2(N)$ and $\T_{Np} := \T_2(Np)$. We will also denote the $p$-th Hecke operator in $\T_{Np}$ by $U_p$ in order to emphasize the different way of acting compared to the one in $\T_N$.

\section{Jacobians of modular curves}
In this section we gather the necessary results from \cite{Ribet90a} that we will need in the proof of the main result.

Let $N$ be a positive integer. Let $X_0(N)_\C$ be the modular curve of level $N$ and $J_0(N):=\operatorname{Pic}^0(X_0(N))$ its Jacobian. There is a well defined action of  the Hecke operators $T_n$ on
$X_0(N)$ and hence, by functoriality, on $J_0(N)$ too. The dual of $J_0(N)$ carries an action of the Hecke algebra as well and can be identified with $S_2(\Gamma_0(N))$.
This implies that one has a faithful action of $\T_N$ on $J_0(N)$. 

Let now $p$  be a prime not dividing $N$. In the same way one has an action of Hecke operators on $X_0(Np)$  and its Jacobian $J_0(Np)$ and the latter admits a faithful action of $\T_{Np}$.
The interpretation of $X_0(N)$ and $X_0(Np)$ allows us to define the two natural degeneracy maps $\delta_1, \delta_p:X_0(Np)\lra X_0(N)$ and their pullbacks $\delta_1^*, \delta_p^*:J_0(N)\lra J_0(Np)$.

There is a map
\begin{equation}\label{seq1}
\alpha:J_0(N)\times J_0(N)\lra J_0(Np),\textrm{\ \ \ \ \ \ \ }(x,y)\mapsto \delta_1^*(x)+\delta_p^*(y).
\end{equation}
whose image is by definition the p-old subvariety of $J_0(Np)$. We will denote this by $A$. This map $\alpha$ is \emph{almost} Hecke-equivariant:
\begin{equation}
\alpha\circ T_q = T_q\circ\alpha\textrm{ for every prime }q\neq p , \label{eq:1}
\end{equation}
\begin{equation}
\alpha\circ\fm{T_p}{p}{-1}{0} = U_p\circ\alpha \label{eq:2}
\end{equation}
Of course, the first one makes sense only if one interprets the operator $T_q$ as
acting diagonally on $J_0(N)\times J_0(N)$.
Consider also the kernel $\Sh$ of the map $J_0(N) \lra J_1(N)$ induced by $X_1(N)\lra X_0(N)$. If we inject it into $J_0(N)\times J_0(N)$ via
$x \mapsto (x, -x)$ then its image, which we will denote by $\Sigma$, is the kernel of the previous map $\alpha$ (see Proposition 1 in \cite{Ribet90a}).
Furthermore $\Sh$, and therefore $\Sigma$ too, are annihilated by the operators $\eta_r = T_r - (r+1)\in \T_N$ for all primes $r\nmid Np$. (see Proposition 2 in \cite{Ribet90a}).

We make a small parenthesis here to introduce a useful notion.
\begin{df}
A maximal ideal $\mg$ of the Hecke algebra $\T_N$ is called Eisenstein if it contains the operator $T_r - (r+1)$ for almost all primes $r$.
\end{df}

We need a few more definitions and facts (see Corollary in \cite{Ribet90a} and the discussion after that):

Let $\Delta$ be the kernel of $\fm{1+p}{T_p}{T_p}{1+p}\in M^{2\times2}(\T_N)$ acting on $J_0(N)\times J_0(N)$. $\Delta$ is finite and comes equipped with a perfect $\mathbb{G}_m$-valued skew-symmetric pairing $\omega$. Furthermore $\Sigma$ is a subgroup of $\Delta$, self orthogonal, and $\Sigma \subset \Sigma^\perp \subset \Delta$. One can also see $\Delta/\Sigma$ and therefore its subroup $\Sigma^\perp/\Sigma$, as a subgroup of $A$.

Let B be the $p$-new subvariety of $J_0(Np)$. It is a complement of $A$, i.e. $A+B = J_0(Np)$  and $A\cap B$ is finite. The Hecke algebra acts on it faithfully through its $p$-new quotient and it turns out (see Theorem 2 in \cite{Ribet90a}) that
\begin{equation}\label{iso:1}
A\cap B \isom \Sigma^\perp/\Sigma.
\end{equation}
as groups, with the isomorphism given by the map $\al$.

\section{Proof of Theorem \ref{thm:main}}

Let $\theta:\T_N\lra \qtn$ be the eigenform associated with $\rho$, $\bar{\theta}:\T_N\lra\Oc/\la$ its reduction $\mod \la$ (which is associated with $\rb$, the $\mod \la$ reduction of $\rho$) and let $I$ and $\mg$ be the kernels of $\theta$ and $\bar{\theta}$ respectively.
It will be enough to find a weak modular form $\theta':\T_2(Np)\lra\qtn$ that agrees with $\theta$ on $T_q$ for all primes $q\neq p$ (i.e. they define the same Galois representation) and factors through $\T_2^{p-\textrm{new}}(Np)$(i.e. new at $p$). In what follows we will be writing $\Ann(M)$ instead of $\Ann_{\T_N}(M)$ to denote the annihilator of a $\T_N$-module $M$. 

Let us begin with the following auxiliary result:
\begin{lem}\label{lem:aux}
$\mg$ is the only maximal ideal of $\T_N$ containing $I$.
\end{lem}

\begin{proof}
We will equivalently show that $\T_N/I$ is local. The proof actually works for any Artinian ring injecting into a local ring.

By the definition of $I$, $\T_N/I$ injects in $\qtn$. Since $\T_N/I$ is Artinian it decomposes into the product of its localizations at its prime (actually maximal) ideals, which are finitely many, say $s\geq 1$. The set containing the identity $e_i$ of each component then forms a complete set (i.e. $\sum_{i=1}^se_i=1$) of pairwise orthogonal (i.e. $e_ie_j=0$ for $1\leq i\neq j\leq s$) non-trivial (i.e. $e_i\neq0,1$) idempotents for $\T_N/I$. The set $\{\bar{e}_1,\ldots, \bar{e}_s\}$ of their image through the injection of $\T_N/I$ into $\qtn$ is clearly a complete set of pairwise orthogonal non-trivial idempotents too. This implies that $\qtn$ is isomorphic to $\prod_{i=1}^s\bar{e}_i(\qtn)$. But this cannot happen unless $s=1$ since $\qtn$ is local. Since $s=1$ we get that $\T_N/I$ is local.
\end{proof}

We define:
$$V_I = J_0(N)[I],$$
$$V_{\mg} = J_0(N)[\mg]$$

We have that $\mg\subseteq\Ann(V_{\mg})$ by the definition of $V_{\mg}$. But $\mg$ is maximal so $\mg=\Ann(V_{\mg})$. We also have that $\Ann(V_I) \subseteq \Ann(V_{\mg}) = \mg$, so $\mg$ is in the support of $\Ann(V_I)$. Since the representation $\rb$, which is the reduction of $\rho$ and it is associated to $\bar{\theta}$, is irreducible we get that $\mg$ is not Eisenstein (See for example Theorem 5.2c in  \cite{Ribet90b}).
Since $I\subseteq\Ann(V_I)$, Lemma \ref{lem:aux} implies that $\Supp(V_I)$ is the singleton $\{\mg\}$.

As in \cite{Ribet90a} we will consider the case where $\tr(\rho(\Frob_p)) \equiv -(p+1) \mod \la^n$. The other case where $\tr(\rho(\Frob_p)) \equiv p+1 \mod \la^n$ is treated in exactly the same case, with some minor alterations which we explicitly mention. Since $\rho$ is modular, associated with $\theta$, this translates to
\begin{equation}\label{eq:3}
\theta(T_p) \equiv -(p+1) \mod \la^n.
\end{equation}
Now consider the composite map
$$J_0(N)\to J_0(N)\times J_0(N)\xrightarrow{\al} A\subseteq J_0(Np),$$
where the first map is the diagonal embedding (in the case of $\tr(\rho(\Frob_p)) \equiv p+1 \mod \la^n$ we pick the anti-diagonal map) and the second is the map $\alpha$ defined in the previous section. By abuse of notation, we will also denote by $V_I$ the image of $V_I$ in $J_0(N)\times J_0(N)$ via the diagonal embedding. We then claim that its intersection with $\Sigma$ is zero:
Assume that it is not, and denote it by $V'_I$. 
It is easy to see that $V'_I$ is preserved by the action of $\T_N$ so it can be seen as a $\T_N$-module: For an $(x,x)\in V'_I$ we have (using realtion \eqref{eq:1})
\begin{equation}\label{eq:4}
\al(T_q(x,x))= T_q(\al(x,x))=T_q(0)=0\qquad\textrm{for primes }q\neq p
\end{equation}
and (using relation \eqref{eq:3})
$$\al(T_p(x,x))=\al(T_p(x),T_p(x))=\al(-(p+1)x,-(p+1)x)=-(p+1)\al(x,x)=0.$$
In the case where $\theta(T_p)\equiv p+1 \mod\la^n$, the elements of $V'_I$ are of the form $(x,-x)$ but the reasoning is the same.
Since $\Sigma$ is annihilated by almost all operators $T_r - (r+1)$, $V'_I$ is annihilated by almost all of them too. This implies that every maximal ideal containing $\Ann(V'_I)$ is Eisenstein. But $\Ann(V_I)\subseteq \Ann(V'_I)$ so $V_I$ has an Eisenstein ideal in its support. On the other hand the only maximal ideal in the support of $V_I$ is $\mg$ which is non-Eisenstein, so we get a contradiction.
One can therefore see $V_I$ as a subgroup of $A$ and we will abuse notation to denote its image through the above map by $V_I$ too. We have the following Lemma:
\begin{lem}
$V_I$ is stable under the action of $\T_{Np}$ and the action is given by a ring homomorphism $\theta':\T_{Np}\lra \qtn$.
\end{lem}

\begin{proof}
This is nothing but a straightforward calculation: 

First note that the action of $\T_N$ on $V_I$ factors through $T_N/I$ so we obtain a map $\theta(\T_N/I)\lra\End(V_I)$. Let $y$ be a non-trivial element of the image of $V_I$ in $A$. Then there exists $x\in V_I$ such that $\al(x,x)=y$. Let now $q$ be a prime other than $p$. In view of relation \eqref{eq:1} and we have that:
$$T_q(y) = T_q(\al(x,x)) = \al(T_q(x),T_q(x)) = \al(\theta(T_q)x,\theta(T_q)x) = \theta(T_q)\al(x,x) = \theta(T_q)y.$$
For $q=p$ we have (using relation \eqref{eq:2} and \eqref{eq:3}):
$$U_p(y) = U_p(\al(x,x)) = \al(\fm{T_p}{p}{-1}{0}(x,x)^{\textrm{T}}) = \al(T_p(x) + px, -x) =$$
$$= \al(\theta(T_p)x + px, -x) = \al(-x,-x) = -\al(x,x) = -y$$
It turns out that $y$ is an eigenvector and that the action of $\T_{Np}$ on it defines a ring homomorphism $\theta':\T_{Np}\lra\qtn$ via:
$$\theta'(T_q) = \theta(T_q)\qquad\textrm{for all primes }q\neq p\textrm{ and}$$
$$\theta'(U_p) = -1$$

To treat the other case one has to keep in mind for the formulas above that $y=\al(x,-x)$ and proceed in the same way to get the same result except that $U_p(y) = y$ this time and therefore $\theta'(U_p) = 1$.
\end{proof}

\remark Since the $\theta$ and $\theta'$ actually agree on almost all primes, it is clear that they are associated with the same Galois representation, so $\theta'$ is the candidate map we were looking for.

To finish of the proof of the main result it remains to show that the map factors through the $p$-new quotient of the Hecke algebra. To this end, it is enough to show that  $V_I$, when viewed as a subgroup of $J_0(Np)$ is a subgroup of $(A\cap B)$. We again proceed according to Ribet. It is easy to see that $V_I$, when considered as a subgroup of $J_0(N)\times J_0(N)$, is a subgroup of $\Delta$.  Let $\bar{V_I}$ be the image of $V_I$ in $\Delta/\Sigma^\perp$. Then, in view of \eqref{iso:1}, we just need to show that $\bar{V}_I$ is trivial.

First notice that $\bar{V}_I$ is preserved by the action of $\T_N$. For this it is enough to check that if $z\in V_I\cap \Sigma^\perp$ then $T_q(z)\in\Sigma^\perp$ and $T_p(z)\in\Sigma^\perp$ (clearly they will also be in $V_I$). Let $x\in\Sigma$. We then have the following: $\omega(x,T_q(z))=\omega(T^\vee_q(x),z)$. Now the subalgebras of generated by $T_q$ and $T^\vee_q$ are isomorphic (see p444 in \cite{Ribet90b}). Since the subalgebra generated by $T_q$ preserves $\Sigma$ as shown in \eqref{eq:4}, we get that $T^\vee_q(x)\in\Sigma$ and therefore that $\omega(T^\vee_q(x),z)=0$. Finally, using \eqref{eq:3} again, $\omega(x,T_p(z))=\omega(x,-(p+1)z)=-(p+1)\omega(x,z)=0$. 

Now according to Ribet in the proof of Lemma 2 in \cite{Ribet90a}, $\Delta/\Sigma^\perp$ is dual to $\Sigma$ which is annihilated by almost all operators $T_r - (r+1)$, so $\Delta/\Sigma^\perp$, and therefore $\bar{V}_I$, is annihilated by them too. This implies that any maximal ideal containing $\Ann(\bar{V}_I)$ is Eisenstein. Recall that $V_I$ is not Eisenstein. Now assume for contradiction that $\bar{V_I}$ is non-zero. Since $\Ann(\bar{V_I})$ contains $\Ann(V_I)$, we get that the support of $V_I$ also contains Eisenstein ideals. This is the desired contradiction that completes the proof of Theorem \ref{thm:main}.

\textbf{Acknowledgments:} I would like to thank Gabor Wiese for many useful comments and discussions, and for helping me with the proof of Lemma \ref{lem:aux}. During the time this research was done I was being supported by the SPP 1489 of the Deutsche Forschungsgemeinschaft (DFG).

\bibliography{References}
\bibliographystyle{alpha}

\end{document}